\documentclass[10pt, a4paper]{amsart}

\usepackage{latexsym,amsmath,amsfonts,amscd,amssymb, mathrsfs, amsthm}
\usepackage[english]{babel}
\usepackage{appendix}
\usepackage{hyperref}
\usepackage{enumitem}
\usepackage[all]{xy}
\usepackage{faktor}
\usepackage{multirow}
\usepackage{tikz-cd}

\usepackage{comment}

\advance\oddsidemargin by -0.5cm
\advance\evensidemargin by -0.5cm
\advance\topmargin by -1.0cm 
\theoremstyle{plain}  
\newtheorem{theorem}{Theorem}[section]

\newtheorem*{theorem*}{Theorem}
\newtheorem{corollary}[theorem]{Corollary}
\newtheorem{lemma}[theorem]{Lemma}
\newtheorem{proposition}[theorem]{Proposition}

\newtheorem{definition}[theorem]{Definition}

\newtheorem*{claim*}{Claim}

\numberwithin{equation}{section}

\newcommand{\R}{\mathbb{R}}

\newcommand{\Z}{\mathbb{Z}}

\newcommand{\C}{\mathbb{C}}
\newcommand{\Hq}{\mathbb{H}}
\newcommand{\Oc}{\mathbb{O}}

\DeclareMathOperator{\cat}{\mathrm{cat}}

\DeclareMathOperator{\ev}{\mathrm{ev}}

\DeclareMathOperator{\cuplength}{\mathrm{cuplength}}

\def\disfrac#1#2{{\displaystyle{\frac{#1}{#2}  }}}

\newcommand{\bdot}{\boldsymbol{\cdot}} 

\newcommand{\Gr}{\mathrm{Gr}} 
\newcommand{\RP}{\mathbb{R}\mathrm{P}} 

\usepackage{xcolor}
\definecolor{MyBlue}{RGB}{0,0,255}

\definecolor{MyRed}{RGB}{255,0,0}

\definecolor{MyGray}{RGB}{150,60,60}

\def\bw#1\ew{\textcolor{brown}{#1}} 
\def\bb#1\eb{\textcolor{blue}{#1}} 
\def\br#1\er{\textcolor{red}{#1}} 
\def\bm#1\em{\textcolor{magenta}{#1}}
\def\bv#1\ev{\textcolor{olive}{#1}}

\begin{document}
	\title{Maps from Grassmannians of 2-planes to projective spaces}

	\author[R. Brasil,\quad A.C. Ferreira \quad \& \quad L. Vandembroucq]{Ricardo Brasil,\, Ana Cristina Ferreira \, \& \, Lucile Vandembroucq}
	\address{\hspace{-5mm} Ricardo Brasil, Ana Cristina Ferreira \& Lucile Vandembroucq; Centro de Matem\'{a}tica,
		Universidade do Minho,
		Campus de Gualtar,
		4710-057 Braga,
		Portugal} 
	\email {ricardodacostabrasil@gmail.com, anaferreira@math.uminho.pt, lucile@math.uminho.pt }

	\keywords{Grassmannians, Quaternions and octonions, Submersions, Lusternik-Schnirelmann category.}
	
	\subjclass[2020]{57R19, 57S25, 55M30, 53C30.}
    
	
	\date{\today}

	\begin{abstract}
		
		Using quaternions and octonions, we construct some maps from the Grassmannian of $2$-dimensional planes of $\R^n$, $\Gr_2(\R^n)$, to the projective space $\RP^k$, for certain values of $n$ and $k$.  All of our maps induce an isomorphism at the level of fundamental groups, and two of them are shown to be submersions. As an application, we obtain new estimates of the Lusternik-Schnirelmann category of $\Gr_2(\R^n)$ for specific values of $n$.
		
	\end{abstract}
	
	\bigskip

	\maketitle

	
	\section{Introduction}

Let $n> d\geq 1$ be integers. We denote by $\Gr_d(\R^n)$ the real Grassmannian of $d$-dimensional linear subspaces of $\R^n$. As is well known, $\Gr_d(\R^n)$ is a manifold of dimension $d(n-d)$
and, by taking the orthogonal subspace, we have a homeomorphism (in fact, a diffeomorphism) $\Gr_d(\R^n)\cong \Gr_{n-d}(\R^n)$. 
When $d=1$, $\Gr_1(\R^n)$ is exactly the projective space $\RP^{n-1}$. We note that, when $d=2$ and $n=3$, the homeomorphism  $\Gr_2(\R^3)\cong\RP^2$ can be described through the cross product of $\R^3$.

The goal of this article is to construct maps from $\Gr_2(\R^n)$ to a certain projective space $\RP^k$ which induce an isomorphism between the fundamental groups, in short a \textit{$\pi_1$-isomorphism}. Recall that, for $n\geq 3$ and $k\geq 2$, we have $\pi_1(\Gr_2(\R^n))=\pi_1(\RP^k)=\Z_2$. Of course, such a map always exists. We can indeed consider the classifying map of the universal cover $\Gr_2(\R^n)\to B\Z_2=\RP^{\infty}$. This map induces a $\pi_1$-isomorphism and, for dimensional reasons, lifts up to homotopy into the $2(n-2)$-skeleton of $\RP^{\infty}$. We then get a map $\Gr_2(\R^n)\to \RP^{2(n-2)}$ which induces a $\pi_1$-isomorphism. We are therefore mainly interested in constructing maps $\Gr_2(\R^n)\to \RP^{k}$ inducing a $\pi_1$-isomorphism with $2\leq k<2(n-2)$. Remark that this is equivalent to constructing $\Z_2$-equivariant maps from the oriented Grassmannian ${\Gr^+_2}(\R^n)$ to $S^k$, where $S^k$ is given with the antipodal action. In \cite{KS91}, Korba\v{s} 
and Sankaran have discussed the existence of $\Z_2$-equivariant maps between oriented Grassmannians. From their work, we can state the following results regarding the particular case in interest here:
\begin{theorem}[\cite{KS91}] \label{KS}  Let $2^s<n\leq 2^{s+1}$ where $s\geq 1$.
	\begin{enumerate}
		\item There does not exist any map $\Gr_2(\R^n)\to \RP^{k}$ inducing a $\pi_1$-isomorphism if $k<2^{s+1}-2$.
		\item There does not exist any map $\Gr_2(\R^{2^s+1})\to \RP^{2n-5}$ inducing a $\pi_1$-isomorphism.
		\item There exist maps $\Gr_2(\R^4)\to \RP^{2}$ and $\Gr_2(\R^7)\to \RP^{6}$ inducing a $\pi_1$-isomorphism.
		\item There exist maps $\Gr_2(\R^n)\to \RP^{k}$ inducing a $\pi_1$-isomorphism whenever $(n,k)$ belongs to $\{(3,2),(j,6), 3\leq j\leq 6\}$.
	\end{enumerate}
\end{theorem}
The non-existence results are obtained through cohomological arguments. The map $\Gr_2(\R^4)\to \RP^{2}$ is obtained through an identification of $\Gr_2(\R^4)$ with the space $(S^2\times S^2)/(u,v)\sim (-u,-v)$ while the map $\Gr_2(\R^7)\to \RP^{6}$ is constructed using the vector product of $\R^7$. The maps in the last item are immediately obtained from those of item $(3)$ by precomposition with the inclusion $\Gr_2(\R^{n-1})\to \Gr_2(\R^n)$.\\

In this paper, we go further in the positive direction. More precisely, we extend the list of positive cases of Theorem \ref{KS} in the following way:

\begin{theorem} \label{main} 
	There exist maps $\Gr_2(\R^4)\to \RP^{2}$, $\Gr_2(\R^8)\to \RP^{6}$, $\Gr_2(\R^9)\to \RP^{14}$, $\Gr_2(\R^{10})\to \RP^{15}$, $\Gr_2(\R^{12})\to \RP^{17}$ and $\Gr_2(\R^{16})\to \RP^{21}$ inducing a $\pi_1$-isomorphism.
	Any precomposition of such a map with an inclusion $\Gr_2(\R^{n-1})\to \Gr_2(\R^n)$ where $n-1>2$ induces a $\pi_1$-isomorphism.	
\end{theorem}	
We first construct maps $\nu_4:\Gr_2(\R^4)\to \RP^{2}$ and $\nu_8:\Gr_2(\R^8)\to \RP^{6}$ by using quartenions and octonions, respectively. These maps naturally extend the maps $\Gr_2(\R^3)\to \RP^{2}$ and $\Gr_2(\R^7)\to \RP^{6}$ obtained by using vector products. The remaining maps are obtained by appropriate extensions and restrictions in the spirit of \cite{Lam}. We also study the maps $\nu_4$ and $\nu_8$  more deeply.  We develop some results that allow us to understand better the structure of these maps and we establish, in particular, the following result:
\begin{theorem} \label{submersions} 
	The maps $\nu_4:\Gr_2(\R^4)\to \RP^{2}$ and $\nu_8:\Gr_2(\R^8)\to \RP^{6}$ are submersions.
\end{theorem}	
As a consequence of this study, $\nu_4$ is, up to homeomorphism, the same as the map described in \cite{KS91} and is a sphere bundle with fibre $S^2$ while $\nu_8$ is a locally trivial fibration with fibre $\mathbb{C}P^3$.  \\

%

Some of the maps described in Theorem \ref{main} permit us to obtain new results on the Lusternik-Schnirelmann of $\Gr_2(\R^n)$. Recall that the (normalized) Lusternik-Schnirelmann category of $X$, $\cat(X)$, is the least integer $k$ such that there exists a cover of $X$ given by $k+1$ open sets $U_i$ of $X$ which are contractible in $X$, that is, the inclusion $U_i\hookrightarrow X$ is homotopically trivial for each $i=0,...,k$. If no such cover exists, one sets $\cat(X)=\infty$. 

Let $n\geq 3$. To our knowledge, the best estimates of $\cat(\Gr_2(\R^n))$ have been obtained by classical results on the LS-category together with specific results due to Hiller \cite{Hiller} and Berstein \cite{Berstein} more than four decades ago (see Section \ref{sec:LS-cat} for more details) and can be stated as follows. For $2^s<n\leq 2^{s+1}$ where $s\geq 1$, we have
\[\left\{\begin{array}{ll}
\cat(\Gr_2(\R^n))=n+2^s-3 &\mbox{ if } n\in\{2^s+1,2^s+2\}, \\
  n+2^s-3\leq \cat(\Gr_2(\R^n))\leq 2n-5 &\mbox{ if } 2^s+3\leq n\leq 2^{s+1}.
\end{array}\right.\] 
In particular, $\cat(\Gr_2(\R^n))$ is completely known for $3\leq n\leq 6$ (note that $\cat(\Gr_2(\R^2))=0$ since $\Gr_2(\R^2)$ is contractible). However, for $n\in \{7,8\}$, the previous results yield
\[8\leq \cat(\Gr_2(\R^7))\leq 9 \quad \mbox{and} \quad 9\leq \cat(\Gr_2(\R^8))\leq 11.\]

Using some maps from Theorem \ref{main} together with a theorem due to Dranishnikov \cite{Dranish}, we are able to obtain better estimates for some values of $n$. In particular, using the map $G_2(\R^8)\to \RP^6$, we obtain the following two exact values.

\begin{theorem}
We have  $\cat(\Gr_2(\R^7))=8$ and $\cat(\Gr_2(\R^8))=9$.
\end{theorem}

	\section{The map $\nu_4:\Gr_2(\R^4) \longrightarrow \RP^2$}\label{sec:nu4}

As is well-known, the set of quarternions $\Hq$ is equipped with a composition algebra structure over $\R$. The norm of a quaternion $x$ is defined as $||x||^2= x\overline{x}$, where $\overline{x}$ denotes the conjugate of $x$. The scalar product of two quaternions $x,y$ is given by 
    $(x|y) = \frac{1}{2}(x\overline{y}+y\overline{x})=\frac{1}{2}(\overline{x}y+\overline{y}x)$. Also, identifying $\R^3$ with the set of pure quarternions, the usual vector product of two elements $v,v$ of $\R^3$ can be seen as $u\wedge v = \frac{1}{2}(uv-vu)$.

	With the goal of building a map $\Gr_2(\R^4) \longrightarrow \RP^2$ by using the quaternions, we first look at the following one:
	\begin{align*}
		\mu_{\Hq} \colon \Hq \times\Hq &\to \Hq \\
		(x,y) &\mapsto \frac{1}{2}(y \, \Bar{x}-x \, \Bar{y})
	\end{align*}
	
	and note that the image of $\mu_{\Hq}$ is contained in the 3-dimensional subspace of pure quaternions since $\overline{\mu_{\Hq}(x,y)}=-\mu_{\Hq}(x,y)$. This 3-dimensional space will be identified with $\R^3$, as usual. 
	
	We give two alternative characterizations of $\mu_{\Hq}$ that can be easily verified.  
	
	\begin{enumerate}[label=\roman*)]
		\item Seeing $\Hq$ as $\R \oplus \R^3$ and writing $x, y$ as $x_{0}+v, y_{0}+w$ respectively, we have:
		$$\mu_{\Hq}(x,y)=x_{0} \, w-y_{0} \, v+v \wedge w.$$ 
		
		\smallskip
		
		\item Writing $x=(x_{0},x_{1},x_{2},x_{3})$ and $y=(y_{0},y_{1},y_{2},y_{3})$ and considering the Pl\"ucker coordinates $p_{ij}=x_{i}y_{j}-y_{i}x_{j}$ for $0\leq i<j \leq 3$, we get:
		$$\mu_{\Hq}(x,y)=(p_{01}+p_{23},p_{02}-p_{13},p_{03}+p_{12}).$$
	\end{enumerate}
The second characterization will be useful in Sec. \ref{sec:other}.
    
	We can now define the desired map, which we will call $\nu_{4}$, in the expected way:
	\begin{align*}
		\nu_{4} \colon \Gr_{2}(\R^4) &\to \RP^2 \\
		\langle x,y\rangle &\mapsto \left[ \frac{1}{2}(y \, \Bar{x}-x \, \Bar{y}) \right]
	\end{align*}
	Here $\langle x,y\rangle$ is the plane generated by $x,y\in \R^4\cong \Hq$ and $[u]\in \RP^2$ is the class of $u\in \R^3\setminus\{0\}$ with respect to the relation $u\sim \lambda u$, $\lambda \in \R\setminus\{0\}$.  
	In order for this map to be well defined we need to see two things, namely: if we change the basis of the plane of the Grassmannian we should get the same element of the projective space, i.e., a multiple of the first vector and no plane can be sent to the null vector.
	
	\begin{proposition} Let $x,y \in \Hq$, then:
		\begin{enumerate}[label=\roman*)]
			\item $\mu_{\Hq}(\alpha x+ \beta y, \gamma x+ \delta y)=(\alpha \delta-\beta \gamma) \,\mu_{\Hq}(x,y)$, where $\alpha,\beta,\gamma,\delta \in \R$ and $\alpha \delta-\beta \gamma \neq 0$;
			
			\medskip
			
			\item If $\mu_{\Hq}(x,y)=0$ then $x$ and $y$ are linearly dependent.
		\end{enumerate}
		
		\label{nu_well _def}
	\end{proposition}
	
	\begin{proof}
		Let $x,y\in \Hq$.
		\begin{enumerate}[label=\roman*)]
			
			\item Considering $\alpha,\beta,\gamma,\delta$ in the intended conditions, we have:
			$\mu_{\Hq}(\alpha x+ \beta y, \gamma x+ \delta y)=\frac{1}{2}(( \gamma x+ \delta y)(\alpha \Bar{x}+ \beta \Bar{y})-(\alpha x+ \beta y)( \gamma \Bar{x}+ \delta \Bar{y}))=(\alpha \delta-\beta \gamma) \,\frac{1}{2}(y \Bar{x}-x \Bar{y})=(\alpha \delta-\beta \gamma) \,\mu_{\Hq}(x,y)$.
			
			\medskip
			
			\item If $\mu_{\Hq}(x,y)=0$ then $||\mu_{\Hq}(x,y)||=0$. We have that $||\mu_{\Hq}(x,y)||^2= \mu_{\Hq}(x,y)\overline{\mu_{\Hq}(x,y)} = \frac{1}{2}||x||^2||y||^2-\frac{1}{4}((y\overline{x})^2+(x\overline{y})^2)$. Now $(x|y)^2= \frac{1}{4}((y\overline{x})^2+(x\overline{y})^2) + \frac{1}{2} ||x||^2||y||^2$. Thus, 
$||\mu_{\Hq}(x,y)||^2= ||x||^2||y||^2 -(x|y)^2$. Therefore, 
			$||x||^2||y||^2-(x | y)^2=0$ and, consequently,  $ ||x||^2||y||^2 \sin (\measuredangle(x,y))^2=0$.
			This means that one of three factors is 0. In either case $x,y$ are, indeed, linearly dependent.
		\end{enumerate}
		\end{proof}
	
	In conclusion, changing the basis does not change the outcome and no plane has null image because two linearly dependent vectors do not define a plane. In this way, it is now confirmed that $\nu_{4}$ is a map $\Gr_2(\R^4) \longrightarrow \RP^2$. Note also that, since $\mu_{\Hq}$ is continuous, so is $\nu_{4}$.\\
	
	\begin{proposition}\label{nu4-pi1}
		The map $\nu_{4}$ induces a $\pi_1$-isomorphism.
	\end{proposition}

\begin{proof}
	Writing $[x_1:x_2:x_3]$ for an element of $\RP^2$, we can check that the map $\sigma: \RP^2\to \Gr_{2}(\R^4)$ that takes $[x_1:x_2:x_3]$ to the plane of $\R^4$ generated by $(1,0,0,0)$ and $(0,x_1,x_2,x_3)$ is a section of $\nu_{4}$, that is $\nu_4\circ \sigma={\rm Id}$. Consequently, $(\nu_4)_*: \pi_1(\Gr_2(\R^4))\to \pi_1(\RP^2)$ is an epimorphism and an isomorphism since $\pi_1(\Gr_2(\R^4))=\pi_1(\RP^2)=\Z_2$.
\end{proof}

    As mentioned in the introduction, the map $\Gr_2(\R^4) \longrightarrow \RP^2$ considered by Korba\v{s} and Sankaran is obtained through the identification of $\Gr_2(\R^4)$ with the space $P=(S^2\times S^2)/\sim$ where $\sim $ is given by $(u,v)\sim (-u,-v)$. More precisely, this is the map $p: P  \to \RP^2$ induced by the projection on the first factor. Note that the quotient $P$ is an example of \textit{projective product spaces} in the sense of  \cite{Davis}. In the following proposition, we show that, up to an homeomorphism, the maps $p$ and $\nu_4$ agree. 
	 
	\begin{proposition} \label{prop:homeom-hq}
    There exists a homeomorphism $\zeta$ making the following diagram commutative
    \begin{equation}\label{diag:v4}
	\xymatrix{
	 \Gr_2(\R^4) \ar[dr]_{\nu_{4}}  \ar[rr]^{\zeta} && P \ar[dl]^{p}\\
	&\RP^2	
	}
     \end{equation}
 and $\nu_4:\Gr_2(\R^4) \to \RP^2$ is a sphere bundle with fibre $S^2$.
 \end{proposition}

In order to prove the proposition, we would like to start with some considerations.

We have defined our map $\mu_{\Hq}: \mathbb{H} \times \mathbb{H} \longrightarrow \mathbb{H}$, $(x,y) \longmapsto \frac{1}{2}(y\overline{x}-x\overline{y})$, but we can also define its counterpart $\tilde{\mu}_{\Hq}:\Hq  \times \Hq\longrightarrow \Hq$, $(x,y)\longmapsto \frac{1}{2}(\overline{x}y-\overline{y}x)$ with perfectly analogous properties.

Remark that if $(x|y) =0$ then $\mu_{\Hq}(x,y)=y\overline{x}$ and $\tilde{\mu}_{\Hq}(x,y) = \overline{x}y$. Also, if $\Vert x\Vert = \Vert y \Vert =1$  then $\mu_{\Hq}(x,y), \tilde{\mu}_{\Hq}(x,y) \in S^2$, where $S^2$ is the space of pure quartenions of norm $1$. This allows to define a third map on the Stiefel manifold
\begin{equation*}
    V_2(\R^4)=\{(x,y)\in \R^4: (x|y) =0 , \Vert x \Vert = \Vert y \Vert =1  \}
\end{equation*} as follows 
\begin{equation*}
    \begin{array}{llcl}
    \xi: & V_2(\R^4) & \longrightarrow & S^2\times S^2\\
     & (x,y) & \longmapsto & (\mu_{\Hq}(x,y), \tilde{\mu}_{\Hq}(x,y))=(y\overline{x}, \overline{x}y)
    \end{array}.
\end{equation*}

 Let us see that $\xi$ is a surjective map. Considering $(u, v)\in S^2\times S^2$, we wish to find $(x,y)\in V_2(\R^4)$ such that $(y\overline{x}, \overline{x}y) = (u,v)$ or,  equivalently, $(xv\overline{x}, y) = (u,xv)$. Recall that 
 given  $u,v\in S^2$, there is a rotation of $\R^3$ which sends $v$ to $u$ and also that the map 
	\begin{align*}
	\vartheta: \, & S^3 \longrightarrow \mathrm{SO}(3) \\
	& x \longmapsto \vartheta_x: v \mapsto xv\overline{x}
	\end{align*}
	is well-known to be surjective. Thus, indeed, there exists $x\in S^3$ such that $xv\overline{x}=u$ and we can take $y=xv$. Clearly, $\Vert y \Vert =  \Vert x \Vert \Vert v \Vert  =1$ and, also, $(x|y) = \frac{1}{2}(\overline{x}xv+\overline{xv}x)=\frac{1}{2}(v+\overline{v})=0$. So, $(x,y) \in V_2(\R^2)$ and $\xi(x,y) = (u,v)$.
    
    However, $\xi$ is not injective. Let $(x,y),(q,z)\in V_2(\R^4)$ such that $\xi(x,y) =(u,v) = \xi(q,z)$. This implies
    \begin{equation*}
        \begin{cases}
            y\overline{x} = u = z \overline{q} \\
            \overline{x}y = v = \overline{q}z 
        \end{cases} \Rightarrow \begin{cases}
            xv\overline{x} = u = qv\overline{q}\\
            y= xv \\
            z = qv 
        \end{cases}.
    \end{equation*}
From $qv\overline{q}=xv\overline{x}$, then $\overline{x}qv=v\overline{x}q$, and so $\overline{x}q$ commutes with $v$. Thus, $\overline{x}q= \alpha+\beta v$ with $\alpha,\beta\in\R$ such that $\alpha^2+\beta^2 =1$ since also $\Vert \overline{q}x\Vert = 1$. We conclude that $q=x(\alpha+\beta v)=\alpha x +\beta y$ and, recalling that $v^2=-v\overline{v}=-1$, that $z=qv =x(\alpha+\beta v)v = -\beta x + \alpha y$. To summarize
    \begin{equation}
        \begin{pmatrix}
            q \\ z 
        \end{pmatrix} = \begin{pmatrix}
            \alpha & \beta \\
            -\beta & \alpha \end{pmatrix} \begin{pmatrix}
            x \\ y 
        \end{pmatrix}, \quad \alpha^2+\beta^2=1, \label{eq:grass_+}
    \end{equation}
that is, $(q,z)$ and $(x,y)$ differ by a transformation of $\mathrm{SO}(2)$.

\begin{proof}[Proof of Proposition \ref{prop:homeom-hq}]
From the considerations above, we have a well defined map $$\tilde{\xi}:V_2(\R^4)/\mathrm{SO}(2) \longrightarrow S^2\times S^2$$ which is a continuous bijection. It is, moreover, a homeomorphism since $V_2(\R^4)/\mathrm{SO}(2)$ is compact, and $S^2\times S^2$ is Hausdorff. Actually, the quotient $V_2(\R^4)/\mathrm{SO}(2)$ or, more precisely, the quotient $V_2(\R^4)/\sim$ where $\sim$ is the equivalence relation given by Eq. \eqref{eq:grass_+}, is the Grassmannian of oriented two planes of $\R^4$, $\Gr_2^+(\R^4)$, and the homeomorphism $\tilde{\xi}$ can be seen as an explicit expression of the well-known homeomorphism 
$\Gr_2^+(\R^4) \cong S^2 \times S^2$, see for instance \cite[p. 104]{GHV}.

To obtain our Grassmannian $\Gr_2(\R^4)$ we can take the quotient of $\Gr_2^+(\R^4)$ under the equivalence relation $(x,y)\sim (y,x)$. Recall that $P=(S^2\times S^2)/\sim$ with $(u,v)\sim (-u,-v)$. Now, since $\tilde{\xi}(y,x) = (x\overline{y}, \overline{y}x)= (\overline{y\overline{x}}, \overline{\overline{x}y})=-(y\overline{x}, \overline{x}y)=-\tilde{\xi}(x,y)$, then we have a well-defined map 
\begin{equation*}
    \zeta \colon \Gr_2(\R^4) \longrightarrow P
\end{equation*}
which is a homeomorphism. It is clear that the diagram \eqref{diag:v4} commutes. Since the map $p$ is known to be a sphere bundle with fibre $S^2$ (see \cite{Davis}), then so is our map $\nu_{4}$.
\end{proof}

	\section{The map $\nu_8:\Gr_2(\R^8) \longrightarrow \RP^6$}
	
	After the construction of $\nu_{4}$, we naturally proceed in a similar way to construct a map on $\Gr_2(\R^8)$, now using the octonion set $\Oc$ that we identify with $\R^8$ given with its canonical basis denoted by $e_0,...,e_7$. Firstly, we will need to define a multiplication in $\Oc$. There are many possible choices but all of them will lead to analogous results. We will work with the multiplication defined through the following table 
	together with the rule $e_ie_j=-e_je_i$ for $i>j\geq 0$. 
    
	\begin{center}
		\begin{tabular}{|cl|llllllll|}
			\hline
			\multicolumn{2}{|c|}{\multirow{2}{*}{$e_{i} \, e_{j}$}} & \multicolumn{8}{c|}{$e_{j}$}                                                                          \\ \cline{3-10} 
			\multicolumn{2}{|c|}{} & \multicolumn{1}{c|}{$e_{0}$}  & \multicolumn{1}{c|}{$e_{1}$}  & \multicolumn{1}{c|}{$e_{2}$}  & \multicolumn{1}{c|}{$e_{3}$}  & \multicolumn{1}{c|}{$e_{4}$}  & \multicolumn{1}{c|}{$e_{5}$}  & \multicolumn{1}{c|}{$e_{6}$}  & \multicolumn{1}{c|}{$e_{7}$ } \\ \hline
			\multicolumn{1}{|c|}{\multirow{8}{*}{$e_{i}$}}  & $e_{0}$ & \multicolumn{1}{c|}{$e_{0}$} & \multicolumn{1}{c|}{$e_{1}$}  & \multicolumn{1}{c|}{$e_{2}$}  & \multicolumn{1}{c|}{$e_{3}$}  & \multicolumn{1}{c|}{$e_{4}$}  & \multicolumn{1}{c|}{$e_{5}$}  & \multicolumn{1}{c|}{$e_{6}$}  & \multicolumn{1}{c|}{$e_{7}$}  \\ \cline{2-10} 
			\multicolumn{1}{|c|}{}                        & $e_{1}$ & \multicolumn{1}{c|}{$\cdot$}  & \multicolumn{1}{c|}{$-e_{0}$} & \multicolumn{1}{c|}{$e_{3}$}  & \multicolumn{1}{c|}{$-e_{2}$} & \multicolumn{1}{c|}{$e_{5}$}  & \multicolumn{1}{c|}{$-e_{4}$} & \multicolumn{1}{c|}{$-e_{7}$} & \multicolumn{1}{c|}{$e_{6}$}  \\ \cline{2-10} 
			\multicolumn{1}{|c|}{}                        & $e_{2}$ & \multicolumn{1}{c|}{$\cdot$}  & \multicolumn{1}{c|}{$\cdot$}  & \multicolumn{1}{c|}{$-e_{0}$} & \multicolumn{1}{c|}{$e_{1}$}  & \multicolumn{1}{c|}{$e_{6}$}  & \multicolumn{1}{c|}{$e_{7}$}  & \multicolumn{1}{c|}{$-e_{4}$} & $-e_{5}$ \\ \cline{2-10} 
			\multicolumn{1}{|c|}{}                        & $e_{3}$ & \multicolumn{1}{c|}{$\cdot$}  & \multicolumn{1}{c|}{$\cdot$}  & \multicolumn{1}{c|}{$\cdot$}  & \multicolumn{1}{c|}{$-e_{0}$} & \multicolumn{1}{c|}{$e_{7}$}  & \multicolumn{1}{c|}{$-e_{6}$} & \multicolumn{1}{c|}{$e_{5}$}  & $-e_{4}$ \\ \cline{2-10} 
			\multicolumn{1}{|c|}{}                        & $e_{4}$ & \multicolumn{1}{c|}{$\cdot$}  & \multicolumn{1}{c|}{$\cdot$}  & \multicolumn{1}{c|}{$\cdot$}  & \multicolumn{1}{c|}{$\cdot$}  & \multicolumn{1}{c|}{$-e_{0}$} & \multicolumn{1}{c|}{$e_{1}$}  & \multicolumn{1}{c|}{$e_{2}$}  & \multicolumn{1}{c|}{$e_{3}$}  \\ \cline{2-10} 
			\multicolumn{1}{|c|}{}                        & $e_{5}$ & \multicolumn{1}{c|}{$\cdot$}  & \multicolumn{1}{c|}{$\cdot$}  & \multicolumn{1}{c|}{$\cdot$}  & \multicolumn{1}{c|}{$\cdot$}  & \multicolumn{1}{c|}{$\cdot$}  & \multicolumn{1}{c|}{$-e_{0}$} & \multicolumn{1}{c|}{$-e_{3}$} & \multicolumn{1}{c|}{$e_{2}$}  \\ \cline{2-10} 
			\multicolumn{1}{|c|}{}                        & $e_{6}$ & \multicolumn{1}{c|}{$\cdot$}  & \multicolumn{1}{c|}{$\cdot$}  & \multicolumn{1}{c|}{$\cdot$}  & \multicolumn{1}{c|}{$\cdot$}  & \multicolumn{1}{c|}{$\cdot$}  & \multicolumn{1}{c|}{$\cdot$}  & \multicolumn{1}{c|}{$-e_{0}$} & $-e_{1}$ \\ \cline{2-10} 
			\multicolumn{1}{|c|}{}                        & $e_{7}$ & \multicolumn{1}{c|}{$\cdot$}  & \multicolumn{1}{c|}{$\cdot$}  & \multicolumn{1}{c|}{$\cdot$}  & \multicolumn{1}{c|}{$\cdot$}  & \multicolumn{1}{c|}{$\cdot$}  & \multicolumn{1}{c|}{$\cdot$}  & \multicolumn{1}{c|}{$\cdot$}  & $-e_{0}$ \\ \hline
		\end{tabular}
	\end{center}

	\medskip The element $e_0$ is then the unit of $\Oc$.
	As usual, the conjugate of $x=\sum x_ie_i$ is $\Bar{x}= x_0e_0- \sum\limits_{i>0} x_ie_i$ and $x$ is a pure octonion if $\Bar{x}=-x$. 

     Analogously to the quartenions, $\Oc$ is also equipped with a composition algebra structure over $\R$ (in fact, there exist only four composition algebras of $\R$ which  are $\R, \C, \Hq$ and $\Oc$). The norm of an octonion $x$ is defined as $||x||^2= x\overline{x}$ and the scalar product of two octonions $x,y$ is given by 
    $(x|y) = \frac{1}{2}(x\overline{y}+y\overline{x})=\frac{1}{2}(\overline{x}y+\overline{y}x).$ Also, identifying $\R^7$ with the set of pure octonions, we can define a vector product in $\R^7$ by setting $u\wedge v = \frac{1}{2}(uv-vu)$, for $u,v\in \R^7$. Finally, we recall that, while $\Oc$ is not associative, any subalgebra $\R(x,y)$ of $\Oc$ generated by two octonions $x,y$ is associative.

	Let $\mu_{\Oc}$ be defined as:
	\begin{align*}
		\mu_{\Oc} \colon \Oc \times \Oc &\to \Oc \\
		(x,y) &\mapsto \frac{1}{2}(y \Bar{x}-x \Bar{y})
	\end{align*}
	
	The image of this map is contained in the 7 dimensional subspace of $\Oc$ of pure octonions that can be identified with $\R^7$. Analogously to $\mu_{\Hq}$, we can write $\mu_{\Oc}(x_0+v,y_0+w)=x_{0}w-y_{0}v+v \wedge w$ if we see $\Oc$ as $\R \oplus \R^7$. Using Pl\"ucker coordinates and writing $x=\sum x_ie_i$, $y=\sum y_i e_i$, we can also describe $\mu_{\Oc}$ as a linear combination of $p_{ij}=(x_iy_j-x_jy_i)$ as follows: $\mu_{\Oc}(x,y)=(p_{01}+p_{23}+p_{45}-p_{67},p_{02}-p_{13}+p_{46}+p_{57},p_{03}+p_{12}+p_{47}-p_{56},p_{04}-p_{15}-p_{26}-p_{37},p_{05}+p_{14}-p_{27}+p_{36},p_{06}+p_{17}+p_{24}-p_{35},p_{07}-p_{16}+p_{25}+p_{34})$.\\
	
	As in Proposition \ref{nu_well _def}, we have: 
	
	\begin{proposition} Let $x,y \in \Oc$, then:
		
		\begin{enumerate}[label=\roman*)]
			\item $\mu_{\Oc}(\alpha x+ \beta y, \gamma x+ \delta y)=(\alpha \delta-\beta \gamma) \,\mu_{\Oc}(x,y)$, where $\alpha,\beta,\gamma,\delta \in \R$ and $\alpha \delta-\beta \gamma \neq 0$;
			
			\medskip
			
			\item If $\mu_{\Oc}(x,y)=0$ then $x$ and $y$ are linearly dependent.
		\end{enumerate}
		
		\label{nu8_well _def}
	\end{proposition}
	
	\begin{proof}
	
	 The proof is perfectly analogous to that of Proposition \ref{nu_well _def} taking into account that $\R(x,y)$ is associative. 
			\end{proof}
	
	As before this gives rise to the following map $\Gr_2(\R^8) \longrightarrow \RP^6$:
	\begin{align*}
		\nu_{8} \colon \Gr_{2}(\R^8) &\to \RP^6\\
		\langle x,y\rangle &\mapsto \left[ \frac{1}{2}(y \Bar{x}-x \Bar{y})\right]
	\end{align*}
	
		\begin{proposition}\label{nu8-pi1}
		The map $\nu_{8}$ induces a $\pi_1$-isomorphism.
	\end{proposition}
	
	\begin{proof} The proof is similar to the proof of Proposition \ref{nu4-pi1}.
	\end{proof}

As we will see in the next section, $\nu_8$ is a submersion and a locally trivial fibration. Here we construct an explicit homeomorphism which permits us to identify the fibre of $\nu_8$. Think of $S^7$ and $S^6$ as the spaces of unitary octonions and unitary pure octonions, respectively. For $u\in S^6$, we denote by $S^1_u$ the circle of $S^7$ lying in the plane $\langle 1,u\rangle$, that is,
$$S^1_u=\{w_u=\alpha+\beta u ~|~ \alpha,\beta\in \R, \alpha^2+\beta^2=1\}$$
Considere the space $$Q = (S^7 \times S^6)/\sim$$ with the equivalence relation $\sim$  given by $(x,u)\sim (\omega_u\, x, \pm u)$. 

	\begin{proposition}\label{prop:v8-fiber}
		The map $Q\longrightarrow \Gr_2(\R^8)$ defined by $(x,u) \mapsto \langle x,ux\rangle$ is a homeomorphism. Furthermore, this map makes the following diagram commute:
		\[\xymatrix{ 
			Q \ar[rr]^-{\xi} \ar[rd]_-{pr} & & \Gr_2(\R^8) \ar[ld]^-{\nu_8}\\
			& \RP^6 & 
		}\]
		where $pr([x,u])=[u]$ and, for any $[u]\in \RP^6$, the fibre $\nu_8^{-1}([u])$ is homeomorphic to $\C\mathrm{P}^3$.
	\end{proposition}

 \begin{proof} 
 Firstly, we will show that the map $\xi$ is well-defined. Let $[x,u]\in Q$. The pair $(x, ux)$ spans a plane in $\Gr_2(\R^8)$, since $(x|ux)=0$. That $x$ and $ux$ are orthogonal is readily seen from the associativity of $\R(u,x)$ and from the fact that $\overline{u}=-u$. Now, let $(\omega_u\, x, \pm u) \in [x,u]$. We aim to see that $\xi([\omega_u\, x, \pm u])=\xi([u,x])$. As $\omega_u \in S^1_u$, there exist $\alpha, \beta \in \R$ such that $\omega_u=\alpha+\beta u$ and $\alpha^2+\beta^2=1$. We have  $\xi([\omega_u\, x, \pm u])=\langle\omega_u \, x, \pm u \, (\omega_u \, x)\rangle = \langle(\alpha+\beta u)\, x, \pm u\, ((\alpha+\beta u)\, x)\rangle =\,  \langle\alpha x + \beta \, u x, \pm \alpha\, u x \pm \beta u^2 x\rangle$. But $u$ is a unitary and pure octonion, so $u^2=-1$. Then $\xi([\omega_u\, x, \pm u]) = \langle\alpha x + \beta \, u x, \pm \alpha\, u x \mp \beta\, x\rangle=\langle x,ux\rangle=\xi([x,u])$. 
 
We will now show that $\xi$ is surjective. Let $\langle x,y\rangle\in \Gr_2(\R^8)$. We can choose $x$ and $y$ such that $(x|y)=0$ and $\Vert x \Vert = \Vert y \Vert = 1$.
Consider $(x,y\overline{x})$ and let us verify that $(x,y\overline{x})\in S^7\times S^6$. That $x\in S^7$ is immediate. For $y\overline{x}$, we have that $\Vert y\overline{x} \Vert = \Vert y\Vert \Vert \overline{x} \Vert =1$. Also, $y\overline{x} + \overline{y\overline{x}} = y\overline{x}+x\overline{y} =0$, since $(x|y)=0$ and, thus, the real part of $y\overline{x}$ is zero. Then, $y\overline{x}\in S^6$.
Observe that $\xi([x,y\overline{x}])= \, \langle x, (y\overline{x})x \rangle = \langle x, y\rangle$ since $\R(x,y)$ is associative.

We now proceed to showing that $\xi$ is injective.  Let $[x,u],[y,v] \in Q$. Suppose that $\xi([x,u])=\xi([y,v])$, i.e., $\langle x,ux\rangle=\langle y,uy\rangle$ and let us see that $[x,u]=[y,v]$, i.e., $(y,v)=(\omega_u \, x,\pm u)$, for some $\omega_u \in S^1_u$. 	Since $y, vy \in \langle x,ux\rangle$, there exist $\alpha, \beta, \gamma, \delta \in \R$, such that $y=\alpha x+\beta u x$ and $vy=\gamma x+\delta u x$. Thus, $y=(\alpha+\beta \, u) x$ and, since $\Vert y \Vert =1$, then $\alpha^2+\beta^2=1$. Consequently, $y=\omega_u\, x$, $\omega_u \in S^1_u$. Now, $vy=(\gamma+\delta\, u)\, x$, so
\begin{equation*}
   v = (vy)\overline{y}=((\gamma+\delta u) x)\overline{((\alpha+\beta u) x)} =((\gamma+\delta u)x)(\overline{x}(\alpha+\beta \overline{u}))  
				= \alpha\gamma + \beta\delta + (\alpha\delta -\beta\gamma) u. 
\end{equation*}
		Since $v$ is a pure octonion, $\alpha\gamma + \beta\delta=0$ and $v=(\alpha\delta -\beta\gamma)\, u$. As $||v||=1$, $\alpha\delta -\beta\gamma=\pm 1$ and $v=\pm u$.
			Therefore, $(y,v)=(\omega_u \, x,\pm u)$.

    Clearly, the map $\xi$ is continuous. Since $Q$ is a compact space and $\Gr_2(\R^4)$ is Hausdorff then $\xi$ is a homeomorphism.

    Let us now quickly verify that the diagram commutes. Take $[x,u] \in Q$, we have $\nu_8 \circ \xi([x,u])=\nu_8(\langle x,ux\rangle)=\left[\frac{1}{2}((ux)\overline{x}-x(\overline{ux}))\right] =\left[\frac{1}{2}(u+u)\right]=[u]=pr([x,u])$.

   It remains to see that, for every $[u]\in \RP^6$, the fibre $\nu_8^{-1}([u])$ is homeomorphic to $\C\mathrm{P}^3$.   From what we proved above, we have that $\nu_8^{-1}([u])$ is homeomorphic to $pr^{-1}([u])$. Now, $pr^{-1}([u]) =  S^7/\sim$ with $\sim$ the equivalence relation given by $x \sim \omega_u x$. We can then write  $\nu_8^{-1}([u]) = S^7/S^1_u$ which is, thus, homeomorphic to $\C\mathrm{P}^3$. 
 \end{proof}
	
\section{Submersions via equivariance}

In this section, we shall make use of the differentiable structure of Grassmanians and, moreover, of the fact that they are homogeneous manifolds. 

The goal is to prove that our maps $\nu_4: \mathrm{Gr}_2(\R^4) \longrightarrow \R\mathrm{P}^2$ and $\nu_8: \mathrm{Gr}_2(\R^8) \longrightarrow \R\mathrm{P}^6$  are submersions. Since they are both clearly smooth and surjective, it suffices to prove that they have constant rank (from surjectivy we get that the rank is maximal). For that we will use the equivariant rank theorem, \cite[Thm. 9.7]{Lee}.

 \begin{theorem}[Equivariant rank theorem]
      Let $M$ and $N$ be smooth manifolds and let $G$ be a Lie group. Suppose $f: M \longrightarrow N$ is a smooth map that is equivariant with respect to a transitive smooth $G$-action on M and any smooth $G$-action on N. Then $f$ has constant rank. 
 \end{theorem}

\subsection{The map $\nu_4:\mathrm{Gr}_2(\R^4) \longrightarrow \R\mathrm{P}^2$}

Consider $\mathrm{Sp(1)} = \{h \in  \mathbb{H}: \, \Vert h \Vert =1 \}$ equipped with quarternionic multiplication. We take two copies of $\mathrm{Sp}(1)$, say $\mathrm{Sp}(1)_L$ and $\mathrm{Sp}(1)_R$,  inside $\mathrm{GL}(4, \R)$  by means of left and right multiplication. More precisely:
\begin{equation}\label{eq:SP1L}
    \mathrm{Sp}(1)_L = \{\varphi \in \mathrm{GL}(4,\mathbb{R}): \varphi(x) = gx, \text{ for some } g \in \mathbb{H} \text{ with } \Vert g \Vert = 1   \}
\end{equation}
\begin{equation}\label{eq:SP1R}
    \mathrm{Sp}(1)_R = \{\psi \in \mathrm{GL}(4,\mathbb{R}): \psi(x) = x\overline{h}, \text{ for some } h \in \mathbb{H} \text{ with } \Vert h \Vert = 1   \}
\end{equation}

Remark that the linear maps  in \eqref{eq:SP1L} and \eqref{eq:SP1R} are orthogonal with respect to the quaternionic inner product on $\mathbb{R}^4$. The Lie group 
$$\begin{array}{lcl}
G   =  \mathrm{Sp}(1)\mathrm{Sp}(1) & = & \{\varphi\psi: \,\, \varphi \in \mathrm{Sp}(1)_L, \psi \in \mathrm{Sp}(1)_R\} \\  & = &  \{\alpha \in \mathrm{GL}(4,\mathbb{R}): \alpha(x) = gx\overline{h}, \text{ for some } g,h \in \mathbb{H} \text{ with } \Vert g \Vert = \Vert h \Vert = 1   \} 
\end{array}$$

 is easily seen to be isomorphic to $(\mathrm{Sp}(1) \times \mathrm{Sp}(1))/\mathbb{Z}_2$, with $\mathbb{Z}_2 = \{(1,1), (-1,-1)\}$, and is, moreover, known to be isomorphic to $\mathrm{SO}(4)$. Consider also 
 \begin{equation*}
     H =  \{\beta \in \mathrm{GL}(4,\mathbb{R}): \beta(x) = gx\overline{g} \text{ for some } g \in \mathbb{H} \text{ with } \Vert g \Vert = 1   \}.
 \end{equation*}
The Lie group $H$ is  isomorphic to $\mathrm{Sp}(1)/\mathbb{Z}_2$, with $\mathbb{Z}_2 = \{1,-1\}$, and, what is more, to $\mathrm{SO}(3)$. 
 Notice that, under the isomorphisms above, the map $G \longrightarrow H$ given by $[g,h] \longmapsto [g]$  is well-defined and is a Lie group homomorphism.
 
 The standard action of $\mathrm{SO}(4)$ on the Stiefel manifold $V_2(\R^4)$ is transitive. This is a well-known fact but can be readily seen by observing that $\mathrm{SO}(4)$ acts transitively on $S^3$ with stabilizer $\mathrm{SO}(3)$ and that $\mathrm{SO}(3)$ acts transitively on $S^2$.  This action, being linear, descends to a transitive action on $\mathrm{Gr}_2(\R^4)$. Furthermore, $\mathrm{SO}(4)$ acts on $S^2$ and, thus, on $\R\mathrm{P}^2$, by means of the standard action of $\mathrm{SO}(3)$ and the homomorphism above.

\begin{proposition}
    The map $\nu_4:\mathrm{Gr}_2(\R^4) \longrightarrow \R\mathrm{P}^2$ is $\mathrm{Sp}(1)\mathrm{Sp}(1)$-equivariant. 
\end{proposition}

\begin{proof}
  For $\alpha \in \mathrm{Sp}(1)\mathrm{Sp}(1)$, with $\alpha: \mathbb{H} \longrightarrow \mathbb{H}$, such that $\alpha(x) = gx\overline{h}$, we have 
\begin{equation*}
       \nu_4 (\alpha\langle x, y \rangle )   =  \nu_4(\langle gx\overline{h}, gy\overline{h}\rangle)    =   \frac{1}{2} ((gy\overline{h})(\overline{gx\overline{h}})-(gx\overline{h})(\overline{gy\overline{h}})).
\end{equation*}
Using the properties of the conjugation and the associativity of $\mathbb{H}$, we then obtain
\begin{equation*}
     \nu_4 (\alpha\langle x, y \rangle )   = \frac{1}{2} (gy(\overline{h}h)\overline{gx} - gx(\overline{h}h)\overline{gy})  
      =  \frac{1}{2} (g(y\overline{x}-x\overline{y})\overline{g}) \\
 =  \alpha\nu_4(\langle x, y \rangle)
\end{equation*}
for every $\langle x, y \rangle \in \mathrm{Gr}_2(\mathbb{R}^4)$.

\end{proof}

\begin{corollary}
      The map $\nu_4:\mathrm{Gr}_2(\R^4) \longrightarrow \R\mathrm{P}^2$ is a submersion.
\end{corollary}

 \subsection{The map $\nu_8:  \mathrm{Gr}_2(\R^8) \longrightarrow \R\mathrm{P}^6$}

As can be expected, a simple adaptation of the argument used for $\nu_4$, taking octonions instead of quaternions, will not work, due to the lack of associativity of octonionic multiplication. 

Our main tool will be the triality theorem of $\mathrm{Spin}(8)$; the Lie group acting transitively on $\mathrm{Gr}_2(\mathbb{R}^8)$ will be, in fact,  $\mathrm{Spin}(7)$ seen as a subgroup of $\mathrm{Spin}(8)$. We will present a brief account here, and, for further details, refer the reader to the beautiful exposition in \cite[Chapter 14]{Harvey}, where $\mathrm{Spin}(8)$ and its three inequivalent 8-dimensional representations are constructed making use of the octonions.

Let $V\subset \mathrm{End}_{\mathbb{R}}(\mathbb{O}\oplus \mathbb{O})$ denote the 8-dimensional Euclidean space defined by 
\begin{equation*}
V= \left\{ A(u) =  \begin{pmatrix} 0 & L_{u}\\ -L_{\bar{u}} & 0 \end{pmatrix} \, : \, u \in \mathbb{O} \right\}
\end{equation*}    
where $L_u$ denotes the left multiplication by $u$, and the norm on $V$ is defined by $\Vert A(u)\Vert =\Vert u \Vert.$

The identification of $V$ with $\mathbb{O}$ as Euclidean spaces allows for a very concrete description of $\mathrm{Spin}(8)$ as a matrix group inside $\mathrm{SO}(8)\times \mathrm{SO}(8)$ which we will adopt as our working definition.

\begin{definition}
A linear endomorphism $g=\begin{pmatrix}
    g_+ & 0 \\
    0 & g_-
\end{pmatrix} \in \mathrm{End}_\mathbb{R}(\mathbb{O}) \oplus \mathrm{End}_\mathbb{R}(\mathbb{O})$ is an element of $\mathrm{Spin}(8)$ if and only if $g_+,g_- \in \mathrm{SO}(8)$ and $\chi_g(u) \in \mathbb{O}  \,  (\text{with } \mathbb{O}\simeq V)$, for all $u\in \mathbb{O}$, where $\chi_g(u) = gA(u)g^{-1}$. 
\end{definition}

We remark that, for every $g \in \mathrm{Spin}(8)$, the map $\chi_g: \mathbb{O}\longrightarrow \mathbb{O}$ is in fact an orthogonal map and $\chi: \mathrm{Spin}(8) \longrightarrow \mathrm{SO}(8)$ is the vector representation of $\mathrm{Spin}(8)$. Moreover, we have an exact sequence of groups 
\begin{equation*}
1 \longrightarrow \mathbb{Z}_2 = \{ 1, -1 \} \longrightarrow \mathrm{Spin}(8) \stackrel{\chi}{\longrightarrow}\mathrm{SO}(8){\longrightarrow} 1
\end{equation*}
which means that $\chi$ is a double cover map. Letting $g_0=\chi_g$, and since $g$ determines $g_0$, we will use the notation $g=(g_+, g_-, g_0)$ for an element in $\mathrm{Spin}(8)$.

We warn the reader of the difference of convention in \cite[Def. 14.6]{Harvey}. 
This choice is such that the triality theorem is more adapted to our purposes. Compare with \cite[Thm. 14.19]{Harvey}
and the explanation in \cite[p. 279]{Harvey}.

\begin{theorem}[Triality theorem]
Let $g_+, g_-$ and $g_0$ be orthogonal linear maps $\mathbb{O} \longrightarrow \mathbb{O}$ with respect to the octonionic inner product.  The triplet $(g_+, g_-, g_0) \in \mathrm{Spin}(8)$ (that is, $g_0=\chi_{(g_+,g_-)}$) if and only if 
\begin{equation*}
    g_+(xy) = g_0(x)g_-(y),
\end{equation*}
 for every $x,y \in \mathbb{O}$.
\end{theorem}


\begin{corollary}\label{cor:triality}
 For a triplet $(g_+, g_-, g_0) \in \mathrm{Spin}(8)$, we have 
\begin{equation*}
    g_0(x\overline{y}) = g_+(x)\overline{g_-(y)} 
\end{equation*}
 for every $x,y \in \mathbb{O}$. 
\end{corollary}

\begin{proof} Let $(g_+, g_-,g_0) \in \mathrm{Spin}(8)$, $x,y \in \mathbb{O}$ and consider $z=x\overline{y}$. We may assume that $y \neq 0$, otherwise the identity holds trivially. 
From the triality theorem, we have $g_+(zy)=g_0(z)g_-(y)$. Thus, $g_0(z) = g_+(zy)(g_-(y))^{-1}$.  Now, since $g_-$ is an orthogonal map, then $\Vert g_-(y) \Vert = \Vert y \Vert$. Hence,
\begin{equation*}
    g_0(z) = g_+(zy)\frac{\overline{(g_-(y))}}{\Vert g_-(y) \Vert^2} = g_+\left( \frac{zy}{\Vert y \Vert^2}\right)\overline{g_-(y)},
\end{equation*}
and recalling $z=x\overline{y}$, we obtain  $g_0(x\overline{y}) = g_+(x)\overline{g_-(y)}$.

\end{proof}

The group $\mathrm{Spin}(7)$ can be seen as a subgroup of $\mathrm{Spin}(8)$ as being the subgroup which fixes a vector $v$ in the vector representation of $\mathrm{Spin}(8)$. Concretely, choosing $v=1$,
\begin{equation*}
    \mathrm{Spin}(7) = \{g\in \mathrm{Spin}(8) : g_0(1) = 1\}.
    \end{equation*}
Using the triality theorem, we have an identification of $\mathrm{Spin}(7)$ inside $\mathrm{SO}(8)$ as
\begin{equation*}
    \mathrm{Spin}(7)= \{ g = (g_+, g_-, g_0) \in \mathrm{Spin}(8):\, g_+ = g_-\}. 
\end{equation*}

Another important Lie group in this context is $G_2$. This is the automorphism group of the algebra of octonions, more precisely
\begin{equation*}
    G_2=\{\varphi \in \mathrm{GL}(8,\mathbb{R}): \, \varphi(xy) = \varphi(x)\varphi(y), \, \forall x,y\in \mathbb{O} \}.
\end{equation*}

The triality theorem can also be used to give a description of $G_2$ as 
\begin{equation*}
    G_2 = \{g\in \mathrm{Spin}(7):\, g(1) =1\}.
\end{equation*}
 
The action of $\mathrm{Spin}(7)$ on $V_2(\mathbb{R}^8)$ given by 
\begin{equation*}
    g(x,y) = (gx,gy), \, \text{ for } g\in\mathrm{Spin}(7), (x,y)\in V_2(\mathbb{R}^8), 
\end{equation*}
is known to be transitive. This can be seen by showing that $\mathrm{Spin}(7)$ acts transitively on $S^7$ with stabilizer $G_2$ and that $G_2$ acts transitively on $S^6$. This action is linear and  descends to an action on $\mathrm{Gr}_2(\mathbb{R}^8)$. Moreover $\mathrm{Spin}(7)$ acts on $S^6$, and consequently on $\mathbb{R}\mathrm{P}^6$, via its vector representation $\chi: \mathrm{Spin}(7) \longrightarrow \mathrm{SO}(7)$.

\begin{proposition}
    The map $\nu_8:\mathrm{Gr}_2(\R^8) \longrightarrow \R\mathrm{P}^6$ is $\mathrm{Spin}(7)$-equivariant. 
\end{proposition}

\begin{proof}
  For $\alpha = (g, g, g_0) \in \mathrm{Spin}(7)$, we have
\begin{equation*}
     \nu_8 (\alpha\langle x, y \rangle )   =  \nu_8(\langle g(x), g(y))   =   \frac{1}{2} [g(y)\overline{g(x)}-g(x)\overline{g(y)}]
\end{equation*}
and, by means of Cor. \ref{cor:triality}, we can write
\begin{equation*}
    \nu_8 (\alpha\langle x, y \rangle )    =  \frac{1}{2} [g_0(y\overline{x})-g_0(x\overline{y})]  
      =  \frac{1}{2} [g_0(y\overline{x}-x\overline{y})]  = \alpha\nu_8(\langle x, y \rangle)
\end{equation*}
for every $\langle x, y \rangle \in \mathrm{Gr}_2(\mathbb{R}^8)$. 

\end{proof}

\begin{corollary}
      The map $\nu_8:\mathrm{Gr}_2(\R^8) \longrightarrow \R\mathrm{P}^6$ is a submersion. 
\end{corollary}

\subsection{An application}

Recall that in Prop. \ref{prop:homeom-hq} we showed that the map $\nu_4$ is a sphere bundle over $\RP^2$ with fibre $S^2$. In particular, $\nu_4$ is a locally trivial fibration. We conclude this section by proving that our other map $\nu_8$ is also a locally trivial fibration. This is a straightforward application of the well-known Ehresmann's lemma. \cite{Ehr-lemma}.

\begin{theorem}[Ehresmann's lemma]
    Let $M$ and $N$ be smooth manifolds and $f: M \longrightarrow N$ a smooth map. If $f$ is a proper map and also a submersion then $f$ is a locally trivial fibration.
\end{theorem}

The maps $\nu_4$ and $\nu_8$ are clearly smooth and automatically proper since Grassmannians are compact manifolds and projective spaces are Hausdorff. Since we already proved above that $\nu_4$ and $\nu_8$ are both submersions, this gives an alternative proof that $\nu_4$ is a locally trivial fibration, and for $\nu_8$, recalling also Prop. \ref{prop:v8-fiber}, we can state the following.

\begin{proposition}
The map $\nu_8: \Gr_2(\R^8) \longrightarrow \RP^6$ is a locally trivial fibration with fibre $\mathbb{C}\mathrm{P}^3$. 
\end{proposition}

 \section{Other maps}\label{sec:other}
 In this section we construct the remaining maps mentioned in Theorem \ref{main}. We first explain why we cannot expect to generalize the construction of $\nu_4$ and $\nu_8$ to higher dimensions. An important point here is that both maps $\mu_{\Hq}$, $\mu_{\Oc}$ can be written as a linear combination of $p_{ij}$ in such a way that we have  \[||\mu_{\Hq}||^2=\sum_{0 \leq i<j \leq 3}p_{ij}^2  \quad \mbox{and} \quad ||\mu_{\Oc}||^2=\sum_{0 \leq i<j \leq 7}p_{ij}^2.\]

	For instance, for $\mu_\mathbb{H}$ we have that 
	\begin{equation*}
	||\mu_{\Hq}||^2 	 = 		(p_{01}+p_{23})^2+(p_{02}-p_{13})^2+(p_{03}+p_{12})^2 =
				\sum_{0 \leq i<j \leq 3}p_{ij}^2+2(p_{01}p_{23}-p_{02}p_{13}+p_{03}p_{12})
			\end{equation*}
and $p_{01}p_{23}-p_{02}p_{13}+p_{03}p_{12}=0$ from the well-known Pl\"ucker relation. A similar computation holds for $\mu_{\Oc}$. Recall that the Lagrange identity states that for any $x,y \in \R^n$ we have
			 \[||x||^2||y||^2-(x | y)^2=\sum\limits_{0 \leq i<j \leq n}p_{ij}^2\]
			and observe that this gives an alternative insight to the fact that $\mu_{\Hq}$ and $\mu_{\Oc}$ are well-defined.
    
	However, an analogous approach will not work for higher dimensions. Let us clarify:
	
	\begin{enumerate}[label=\alph*)]
		\item On one side, when trying to build a map $\Gr_2(\R^{n}) \longrightarrow \RP^{k}$, we know from Theorem \ref{KS} that for $2^s < n \leq 2^{s+1}$, $s \geq 1$, we cannot expect a map inducing a $\pi_1$-isomorphism for $k < 2^{s+1}-2$.\\
		
		\item On the other side, Hurwitz theorem states that if the equality $||x||^2||y||^2=||z||^2$ stands for every $x,y \in \R^n$ where $z=(z_{1},\ldots,z_{n}) \in \R^n$ and $z_{i}=\sum a_{ijk}x_{i}y_{j}$, $a_{ijk} \in \R$, then $n \in \{1,2,4,8\}$.\\
	\end{enumerate}

	Suppose now that we intend to build a map $\Gr_2(\R^{n}) \longrightarrow \RP^{k}$ for $n=2^m>8$. Then, as $2^{m-1}<n \leq 2^m$, we know by a) that $k \geq 2^{m}-2$. As the lower the $k$ the better, the ideal would be $k=2^{m}-2$. Suppose that there exists a map $\mu_n:\R^n\times \R^n\to \R^{k+1}=\R^{n-1}$ given by $\mu_{n}(x,y)=(\mu_{n}^1(x,y),\ldots, \mu_{n}^{k+1}(x,y))$ where each $\mu_n^l$ is a linear combination of $p_{ij}$ and $||\mu_n||^2=\sum_{0 \leq i<j \leq n-1}p_{ij}^2$.
Then $\mu_{n}^l(x,y)=\sum a_{ijk}x_{i}y_{j}$ for all $l \in \{1,\ldots,k+1\}$.
	Now, let $x,y \in \R^n$ and $z=((x | y),\mu_{n}^1(x,y),\ldots, \mu_{n}^{k+1}(x,y)) \in \R^{k+2}=\R^{n}$. Then $||z||^2=(x | y)^2+||\mu_{n}(x,y)||^2=||x||^2||y||^2$ (by Lagrange's identity). So, by b), $n \in \{1,2,4,8\}$. As $n>8$, we have a contradiction. \\

	This means that a change of strategy will be needed going forward.

	\subsection{The map $\nu_{16}:\Gr_2(\R^{16}) \longrightarrow \RP^{21}$}
	
	Now we are considering $\Gr_2(\R^{16})$ and, consequently, taking two vectors with 16 real entries each. Intuitively, we will ``split each of them in the middle'' and look at the parts as two octonions.
	
	\begin{theorem}\label{nu16}
		Consider $x,y \in \R^{16}$. Let $x=(x_{1},x_{2}), y=(y_{1},y_{2}) \in \R^8 \oplus \R^8$, then:
		\begin{align*}
			\nu_{16} \colon \Gr_2(\R^{16}) &\longrightarrow \RP^{21} \\
			\left\langle
			\begin{pmatrix}
				x_{1} \\
				x_{2}
			\end{pmatrix}
			,
			\begin{pmatrix}
				y_{1} \\
				y_{2}
			\end{pmatrix}
			\right\rangle
			&\longmapsto [\mu_{\Oc}(x_{1},y_{1}),\mu_{\Oc}(x_{2},y_{2}),x_{1}y_{2}-y_{1}x_{2}]
		\end{align*}
		is a map $\Gr_2(\R^{16}) \longrightarrow \RP^{21}$ inducing a $\pi_1$-isomorphism.
	\end{theorem}
	
	\begin{proof}
		Let us prove that this map is well defined. Firstly take another basis for the plane $\langle x,y \rangle$, i.e., $\{ \alpha x + \beta y, \gamma x +\delta y \}$ with $\alpha,\beta,\gamma,\delta \in \R$ such that $\alpha\delta-\beta \gamma\neq 0$.\\
		In terms of $x_{i},y_{i}$ we have to calculate $\nu_{16} \left( \left\langle
		\begin{pmatrix}
			\alpha x_{1} +\beta y_{1}\\
			\alpha x_{2} +\beta y_{2}
		\end{pmatrix}
		,
		\begin{pmatrix}
			\gamma x_{1} +\delta y_{1} \\
			\gamma x_{2} +\delta y_{2}
		\end{pmatrix}
		\right\rangle
		\right)$.\\
		
		From Proposition \ref{nu8_well _def}, we know that: $\mu_{\Oc}(\alpha x_{i} +\beta y_{i},\gamma x_{i} +\delta y_{i})=(\alpha \delta-\beta \gamma)\, \mu_{\Oc}(x_{i},y_{i})$ for $i \in \{1,2\}$.
		Thus it only remains to see what happens for the third component $x_{1}y_{2}-y_{1}x_{2}$. We have:\\
		$(\alpha x_{1} +\beta y_{1})(\gamma x_{2} +\delta y_{2})-(\gamma x_{1} +\delta y_{1})(\alpha x_{2} +\beta y_{2})=(\alpha \delta-\beta \gamma)(x_{1}y_{2}-y_{1}x_{2})$.\\
		So we can put the determinant of the basis change in evidence and get the same element of $\RP^{21}$.\\
		
		Suppose now that $(\mu_{\Oc}(x_{1},y_{1}),\mu_{\Oc}(x_{2},y_{2}),x_{1}y_{2}-y_{1}x_{2})=0$. Then $\mu_{\Oc}(x_{1},y_{1})=0$, $\mu_{\Oc}(x_{2},y_{2})=0$, and $x_{1}y_{2}=y_{1}x_{2}$ and we want to see that $x=(x_1,x_2)$ and  $y=(y_1,y_2)$ are linearly dependent.
		If both $x_{1},x_{2} \neq 0$, then, by Proposition \ref{nu8_well _def}, $y_{1}=\lambda_{1} x_{1}, \, y_{2}=\lambda_{2} x_{2}$ with $\lambda_{1},\lambda_{2} \in \R$ and therefore $x_{1}(\lambda_{2} x_{2})=(\lambda_{1} x_{1})x_{2}$. From this last equation $\lambda_{1}=\lambda_{2}$ and hence $y=\lambda_{1}x$.\\
		If $x_{1}=0$ then $y_{1}x_{2}=0$. We now have two cases:
		
		\begin{itemize}
			\item If $x_{2}=0$ then $x=0$ and $x,y$ are linearly dependent.
			
			\item If $x_2\neq 0$ then $y_{1}=0$ and $y_{2}=\lambda x_{2}$ with $\lambda \in \R$ (because $\mu_{\Oc}(x_{2},y_{2})=0$). We therefore have $y=(0,y_{2})=(0,\lambda x_{2})=\lambda x$.
		\end{itemize}
		
		The case $x_{2}=0$ is analogous. 
		Finally, note that the codomain of the map $$(x,y)\mapsto (\mu_{\Oc}(x_{1},y_{1}),\mu_{\Oc}(x_{2},y_{2}),x_{1}y_{2}-y_{1}x_{2})$$ can be restricted to a space of dimension 22 because the image of $\mu_{\Oc}$ is a 7-dimensional space and the last coordinate lies in an 8-dimensional space. We then get a well-defined map $\nu_{16}:\Gr_2(\R^{16})\to \RP^{21}$.
		
		In order to see that $\nu_{16}$ induces a $\pi_1$-isomorphism, we consider the following commutative diagram
		\[
		\xymatrix{\Gr_2(\R^{8}) \ar[d]_{\nu_8}\ar[r]& \Gr_2(\R^{16}) \ar[d]^{\nu_{16}}\\
			\RP^6\ar[r] & \RP^{21}		
	}
		\]
	in which the top horizontal map is induced by the inclusion $\R^8\to \R^{16}$, $(x_1,y_1)\mapsto ((x_1,0),(y_1,0))$ and the bottom horizontal map is induced by the inclusion $\R^7\to \R^{22}$, $z\mapsto (z,0)$. As is well-known, each of these two horizontal maps induces a $\pi_1$-isomorphism. Since we know that $\nu_{8}$ also induces a $\pi_1$-isomorphism, we can conclude that so does $\nu_{16}$.

	\end{proof}

	\subsection{Extensions and restrictions}\label{sec:res-ext}
	
	We are interested in finding maps, like the ones built before, where the dimension of the projective space is less than or equal to the one of the Grassmannian, that is maps $\Gr_2(\R^n) \longrightarrow \RP^k$ where $k \leq 2(n-2)$ and the lower the $k$ the better. In this section, we describe some maps that are obtained  by restriction and/or extension of the maps already presented. All of them will induce $\pi_1$-isomorphisms, as it can be checked through an argument as in the proof of Theorem \ref{nu16}.\\
	
	Taking restrictions of $\nu_{8}$ the dimension of the codomain does not change, however we still have reasonable dimensions compared to the dimension of the Grassmannian, namely we have maps: $\Gr_{2}(\R^7)\longrightarrow \RP^6$, $\Gr_{2}(\R^6)\longrightarrow \RP^6$, $\Gr_{2}(\R^5)\longrightarrow \RP^6$.\\
	As for $\nu_{16}$ a few restrictions are to notice. Looking at $\R^{16}$ as two copies of $\Oc$ we can restrict the second of these copies to $\Hq$ (considering only the first four elements), in this way $\mu_{\Oc}(x_{1},y_{1})$ is 7-dimensional, $\mu_{\Oc}(x_{2},y_{2})$ is 3-dimensional and $x_{1}y_{2}-y_{1}x_{2}$ remains 8-dimensional. This gives a total of 18 dimensions hence a map $\Gr_{2}(\R^{12})\longrightarrow \RP^{17}$.\\
	With  similar reasoning, restricting the second copy of $\Oc$ to $\C$ and $\R$ we obtain the maps $\Gr_{2}(\R^{10})\longrightarrow \RP^{15}$ and $\Gr_{2}(\R^{9})\longrightarrow \RP^{14}$ respectively. Obviously, as before for $\nu_{8}$, we have the following maps (without lowering the projective space dimension): $\Gr_{2}(\R^{15})\longrightarrow \RP^{21}$, $\Gr_{2}(\R^{14})\longrightarrow \RP^{21}$, $\Gr_{2}(\R^{13})\longrightarrow \RP^{21}$ and $\Gr_{2}(\R^{11})\longrightarrow \RP^{17}$.\\
	As for extensions, following the line of thinking that led to the construction of $\nu_{16}$ we have a map $\Gr_{2}(\R^{24})\longrightarrow \RP^{44}$, 
	$\left\langle
	\begin{pmatrix}
		x_{1} \\
		x_{2} \\
		x_{3}
	\end{pmatrix}
	,
	\begin{pmatrix}
		y_{1} \\
		y_{2} \\
		y_{3}
	\end{pmatrix}
	\right\rangle
	\longmapsto [\mu_{\Oc}(x_{1},y_{1}),\mu_{\Oc}(x_{2},y_{2}),\mu_{\Oc}(x_{3},y_{3}),P_{12},P_{13},P_{23}]$
	where $P_{ij}=x_{i}y_{j}-y_{i}x_{j}$ is an element of $\mathbb{O}$. Notice that $44 = 2(24-2)$ so, in this case, the dimension of the Grassmannian and the projective space coincide.
	Moreover, any similar extension of this map to higher dimensions of the Grassmannian will lead to a too high dimension of the projective space. For instance, already for $n=32$, we get a map $\Gr_2(\R^{32})\longrightarrow\RP^k$ with $k=75 \geq 60$.

	\section{On the Lusternik-Schnirelmann category of $\Gr_2(\R^n)$}\label{sec:LS-cat}
	In this section, we use some of the maps constructed before to obtain new estimates of the Lusternik-Schnirelmannn category of $\Gr_2(\R^n)$ for some specific values of $n$.
    
    We first recall that the Lusternik-Schnirelmann category of a map $f:X\to Y$, $\cat(f)$, is the least integer $k$ such that there exists a cover of $X$ given by $k+1$ open sets $U_i\subset X$ on each of which the map $f_{|U_i}:U_i\hookrightarrow Y$ is homotopically trivial. The LS-category of a space $X$ as defined in the introduction can then be seen as a special case of the category of a map, explicitly $\cat(X)=\cat({\rm Id}_X)$. We also recall that, if $R$ is a ring, a lower bound of $\cat(X)$ is given by the cuplength of $X$ over $R$, $\cuplength_R(X)$, that is, the length of the maximal nontrivial cup-product in $H^+(X;R)$. Moreover, if $X$ is a CW-complex, $\cat(X)$ is bounded above by the dimension of $X$. We then have for any CW-complex $X$
	\[\cuplength_R(X)\leq \cat(X)\leq \dim(X).\]
	We refer to \cite{CLOT} for more information on the LS-category.
    
	Let $n\geq 3$ be an integer. Recall that the dimension of the Grassmannian $\Gr_2(\R^{n})$ is $2(n-2)$. It is known (\cite{Hiller}, \cite{Stong}) that $\cuplength(\Gr_2(\R^n))=n+2^s-3$ when $2^s<n\leq 2^{s+1}$ with $s\geq 1$. If $n=2^{s}+1$ then $\cuplength(\Gr_2(\R^n))=\dim(\Gr_2(\R^n))$ and consequently $\cat(\Gr_2(\R^n))=n+2^s-3$. If $2^{s}+2\leq n \leq 2^{s+1}$, we can use the results of \cite{Berstein} to see that $\cat(\Gr_2(\R^n))\leq \dim(\Gr_2(\R^n))-1$. In particular, if $n=2^{s}+2$ then $\cat(\Gr_2(\R^n))= \dim(\Gr_2(\R^n))-1=n+2^s-3$ and 
	\[n+2^s-3\leq \cat(\Gr_2(\R^n))\leq 2n-5 \quad\mbox{ if } 2^s+3\leq n\leq 2^{s+1}.\]
	
	In order to improve this estimate we use the following theorem established by A. Dranishnikov:
	
	\begin{theorem}\cite{Dranish}\label{Dr} Let $X$ be a path-connected CW-complex with fundamental group $\pi=\pi_1(X)$. Let $\varphi_X:X\to B\pi$ a classifying map for the universal cover (that is, a map which induces a $\pi_1$- isomorphism) then
		\[\cat(X)\leq \disfrac{\cat(\varphi_X)+\dim(X)}{2}.\] 
		\end{theorem}
	
	Specializing to $\Gr_{2}(\R^n)$ we can state:
	
	\begin{proposition} Let $n\geq 3$. If there exists a map $\Gr_2(\R^n)\to \RP^k$ which induces a $\pi_1$-isomorphism, then 
		\[\cat(\Gr_{2}(\R^n))\leq \disfrac{k+2(n-2)}{2}.\] 
	\end{proposition}

\begin{proof} Since $\pi=\pi_1(\Gr_2(\R^n)))=\Z_2$, we have $B\pi=\RP^{\infty}$ and a classifying map for the universal cover of $\Gr_2(\R^n)$ is just a map $\varphi:\Gr_2(\R^n)\to \RP^{\infty}$ which induces a $\pi_1$-isomorphism. Suppose that there exists a map $\Gr_2(\R^n)\to \RP^k$ which induces a $\pi_1$-isomorphism, then by composing with the inclusion $\RP^k \to \RP^{\infty}$ we obtain a classifying map $\varphi:\Gr_2(\R^n)\to \RP^{\infty}$. Now, it is well-known that for maps $f,g$, we have $\cat(f\circ g)\leq \min\{\cat(f),\cat(g)\}$ and that $\cat(f)\leq \min\{\cat(X),\cat(Y)\}$ whenever $f$ is a map from $X$ to $Y$. Since our map $\varphi$ factorizes through $\RP^k$ and $\cat(\RP^k)=k$,  we can conclude that $\cat(\varphi)\leq k$. The statement follows then from Theorem \ref{Dr}. 
	\end{proof} 
	
Using the maps $\Gr_2(\R^8)\to \RP^{6}$, $\Gr_2(\R^{12})\to \RP^{17}$ and $\Gr_2(\R^{16})\to \RP^{21}$ from Theorem \ref{main} and some of their restrictions, we then obtain the following new estimates of $\cat(\Gr_{2}(\R^n))$:

\begin{theorem} We have
		\[\begin{array}{|c|c|c|c|c|c|c|}
	\hline
	n & 7 & 8& 12& 14& 15& 16 \\ 
	\hline
 \cat(\Gr_{2}(\R^n)) & 8 & 9& [17,18]& [19,22]& [20,23]& [21,24] \\
 \hline 
\end{array}\]
	\end{theorem}

	\bigskip
	
	\section*{Acknowledgments} 
	

All authors  acknowledge the support of CMAT (Centro de Matem\'atica da Universidade do Minho). Their research was financed by Portuguese Funds through FCT (Fundação para a Ciência e a Tecnologia, I.P.) within the Projects UIDB/00013/2020 and UIDP/00013/2020 with the references 
DOI: 10.54499/UIDB/00013/2020 (https://doi.org/10.54499/UIDB/00013/2020) and  DOI: 10.54499/UIDP/00013/2020 (https://doi.org/10.54499/UIDP/00013/2020).
LV thanks Mahir Bilen Can for interesting discussions on the theme of this article.
    
	\bigskip

	\bigskip
\end{document}